\documentclass[12pt]{amsart}
\usepackage{amsfonts, amssymb, amsmath, amsthm, euscript}
\swapnumbers
\theoremstyle{plain}
\newtheorem{thm}{Theorem}

\newtheorem{lem}[thm]{Lemma}

\newtheorem{cor}[thm]{Corollary}

\theoremstyle{definition}

\newtheorem*{Ack}{Acknowledgement}
\newtheorem{note}[thm]{}

\theoremstyle{remark}
\newtheorem{rem}[thm]{Remark}

\def\nxn{n{\times}n}
\def\A{\overline{A}}

\def\Sem{\overline{S}}

\def\x{\overline{x}}
\def\y{\overline{y}}

\begin{document}

\title[Free subsemigroups]{On free subsemigroups \\ of associative algebras}

\author{Edward S. Letzter}

\address{Department of Mathematics\\
        Temple University \\ Philadelphia, PA 19122}
      
      \email{letzter@temple.edu}

      \begin{abstract} 
        In 1992, following earlier conjectures of Lichtman and
        Makar-Limanov, Klein conjectured that a noncommutative domain
        must contain a free, multiplicative, noncyclic
        subsemigroup. He verified the conjecture when the center is
        uncountable. In this note we consider the existence (or not)
        of free subsemigroups in associative $k$-algebras $R$, where
        $k$ is a field not algebraic over a finite subfield. We show
        that $R$ contains a free noncyclic subsemigroup in the
        following cases: (1) $R$ satisfies a polynomial identity and
        is noncommutative modulo its prime radical. (2) $R$ has at
        least one nonartinian primitive subquotient. (3) $k$ is
        uncountable and $R$ is noncommutative modulo its Jacobson
        radical. In particular, (1) and (2) verify Klein's conjecture
        for numerous well known classes of domains, over countable
        fields, not covered in the prior literature.
\end{abstract}



\maketitle 

%%%%%%%%%%%%%%%%%%%%%%%%%%%%%%%%%%-Begin-Body-%%%%%%%%%%%%%%%%%%%%%%%%%%%%%%%%%

In 1977, Lichtman conjectured that the group of units of a
noncommutative division algebra always contains a noncyclic free
subgroup \cite{Lic-one}. Since then, an extensive, broad, and ongoing
literature has developed on the existence of free subobjects of
associative algebras. The reader is referred, e.g., to \cite{Gon-Shi}
for an overview of this literature and, e.g., to \cite{Fer-For-Gon}
and references therein for more recent results.

Our focus in this note is on free subsemigroups of associative
algebras. (Henceforth, references to free subsemigroups and free
subgroups will assume them to be multiplicative and noncyclic.) In
1984, Makar-Limanov conjectured that a noncommutative division algebra
must contain a free subsemigroup, and he proved the same for division
algebras over uncountable fields \cite{Mak}. In 1992, Klein proved
that a noncommutative domain with an uncountable center must contain a
free subsemigroup, and he conjectured the same for all noncommutative
domains \cite{Kle}.  Chiba proved in 1995 that a polynomial extension
of a division algebra must contain a free subsemigroup (and also that a
division algebra over an uncountable field must contain a free
subgroup) \cite{Chi}.  In 1996, Reichstein showed that an algebra over
an uncountable field contains a free subsemigroup (or
subalgebra) if the same holds true after an extension of the scalar
field \cite{Rei} (cf.~Smoktunowicz's constructions \cite{Smo}).

Our primary aim, then, is to establish the existence of free
subsemigroups for some well-known classes of algebras not covered in
the prior literature.

\begin{note} \emph{Setup.} Throughout, $k$ will denote a field not
  algebraic over a finite subfield. All mention of algebras, rings,
  subrings, and subalgebras will assume them to be associative and 
  unital. Finitely generated $k$-algebras will be referred to as
  \emph{$k$-affine}.
\end{note} 

  We begin with an elementary but useful ``specialization'' lemma,
  adapted from Passman \cite[\S 2]{Pas-one}.

\begin{lem} \label{subfactor} Let $A$ be a subring of a ring $R$, and
  suppose there exists a surjective ring homomorphism
  $\varphi\colon A \rightarrow \A$. If $\A$ contains a free 
  subsemigroup then so does $R$. 
\end{lem}

\begin{proof} Choose $\x, \y \in \A$ that generate a  free
  subsemigroup $\Sem$ of $\A$. Choose $x \in \varphi^{-1}(\x)$ and
  $y \in \varphi^{-1}(\y)$. Letting $S$ denote the subsemigroup of $A$
  generated by $x$ and $y$, we see that $\varphi$ restricts to a
  surjective semigroup homomorphism, from $S$ onto $\Sem$, mapping $x$ to
  $\x$ and $y$ to $\y$. By universality, $S$ must be isomorphic to
  $\Sem$, and so $S$ is a free  subsemigroup of $R$.
\end{proof}

\begin{note} \label{alternative} Let $n$ be an integer $> 1$. In view
  of the Tits Alternative \cite{Tit}, and recalling that $k$ is not
  algebraic over a finite subfield, we can conclude that $GL_n(k)$
  contains a free subgroup. Therefore, the (full) matrix algebra
  $M_n(k)$ contains a free subgroup, as does the algebra of $\nxn$
  matrices over any $k$-algebra. Lichtman \cite{Lic-two} and
  Gon\c{c}alves \cite{Gon} used the Tits Alternative to verify that a
  noncommutative division algebra finite dimensional over its center
  must contain a free subgroup. We can now conclude that a
  noncommutative central simple algebra, finite dimensional over a
  field extension of $k$, must contain a free subgroup.
\end{note}

The proof of the following employs a reduction to the $k$-affine case,
and I am grateful to Ken Brown for this approach.

\begin{thm} \label{PI} Let $R$ be a $k$-algebra satisfying a
  polynomial identity, and suppose that $R$ is noncommutative modulo
  its prime radical. Then $R$ contains a free subsemigroup.
\end{thm}

\begin{proof} By (\ref{subfactor}), we may assume without loss of generality that
  $R$ is a noncommutative semiprime PI algebra over $k$. Choose
   $x, y \in R$ such that $b := xy-yx \ne 0$. 

  Next, since $R$ is semiprime and PI, it follows that $RbR$ cannot be
  a nil ideal (see, eg., \cite[13.2.6i]{McC-Rob}). Consequently, there
  exist $a_1, c_1, a_2, c_2, \ldots, a_t, c_t \in R$ such that
%%%%%%%%%%%%%%%%%%%%%%%%%%%%%%%%%%%%%%%%%%%5
\[ d := a_1 b c_1 + a_2 b c_2 + \cdots + a_t b c_t \]
%%%%%%%%%%%%%%%%%%%%%%%%%%%%%%%%%%%%%%%%%%%%%%
is not nilpotent.

Set $R' = k\langle x, y, b, a_1, c_1, \ldots, a_t, c_t \rangle$, a
$k$-affine (not necessarily semiprime) PI algebra. Since $d$ is not
nilpotent, the ideal $R'bR'$ of $R'$ is not nil. But it was proved by
Amitsur that the Jacobson radical of a $k$-affine PI algebra must be
nil \cite{Ami}. (Nilpotency was later established by Braun
\cite{Bra}.)  Therefore, $R'bR'$ cannot be contained in the Jacobson
radical of $R'$. Hence $xy-yx$ is not contained in the Jacobson
radical of $R'$, and we can conclude that $R'$ is noncommutative
modulo its Jacobson radical.

We now know that there exists at least one primitive ideal $P'$ of
$R'$ such that $R'/P'$ is noncommutative, and so by Kaplansky's
Theorem (see, e.g., \cite[13.3.8]{McC-Rob}), $R'/P'$ must be a
noncommutative central simple algebra, finite dimensional over a field
extension of $k$. Therefore, by (\ref{alternative}), $R'/P'$ contains a
free subsemisubgroup and, by (\ref{subfactor}), $R$ contains a free
subsemigroup.
\end{proof}

\begin{thm} \label{nonart} Let $R$ be a $k$-algebra with at least one
  nonartinian primitive subquotient. Then $R$ contains a free
  subsemigroup.
\end{thm}

\begin{proof} To start, assume that $R$ is both left primitive and not
  artinian.  Let $M$ be a simple faithful left $R$-module, and
  let $E = \text{End}_RM$.  Since $R$ is not artinian, it follows that
  $M$ has infinite length as a right $E$-module. Also, for any given
  integer $n > 1$, it follows from the Jacobson Density Theorem that
  there exists a $k$-subalgebra $A$ of $R$ equipped with a surjective
  homomorphism $\varphi\colon A \rightarrow M_n(E)$; see, e.g.,
  \cite[11.19]{Lam}. Again by (\ref{alternative}), the subalgebra
  $M_n(k)$ of $M_n(E) = \varphi(A)$ must contain a free
  subsemigroup. The theorem now follows from (\ref{subfactor}).
\end{proof}

\begin{rem} Observe that the preceding two results (\ref{PI} and
  \ref{nonart}) verify Klein's conjecture when $R$ is a nonartinian
  primitive domain or when $R$ is a noncommutative PI domain. Moreover,
  those two results (especially when $k$ is countable) establish the
  existence of free subsemigroups in numerous well known classes of
  algebras not covered in the prior literature. Indeed, algebras with
  a nonartinian primitive subquotient or that are semiprime and PI are
  commonplace among iterated Ore extensions (cf.~\cite{McC-Rob}),
  enveloping algebras of Lie algebras (cf.~\cite{Dix}), group algebras
  (cf.~\cite{Pas-two}), quantum groups (cf.~\cite{Bro-Goo}), and
  algebras arising in nocommutative algebraic geometry
  (cf.~\cite{BRSS}).
\end{rem}

We can also use (\ref{nonart}) to show that Makar-Limanov's
conjecture, over $k$, is equivalent to an \emph{a priori} more general
statement.

\begin{cor} \label{equiv} The following two statements are equivalent: {\rm (i)}
  Every noncommutative division algebra over $k$ contains a free
  subsemigroup. {\rm (ii)} Every $k$-algebra
  noncommutative modulo its Jacobson radical contains a free
  subsemigroup.
\end{cor}

\begin{proof} It suffices to prove that (i) implies (ii). So assume
  (i), and let $R$ be a $k$-algebra noncommutiative modulo its
  Jacobson radical. To prove the corollary it will suffice to show, as
  follows, that $R$ contains a free subsemigroup: First, by
  (\ref{subfactor}) we can assume without loss of generality that $R$
  is primitive and noncommutative. Next, by (\ref{nonart}), we can
  further reduce to the case when $R$ is simple artinian. Now, if $R$
  has rank $n > 1$, then $R$ contains a copy of $M_n(k)$ and so
  contains a free subsemigroup by (\ref{alternative}). It remains only
  to consider the case when $R$ is a division ring, which is exactly
  (i).
\end{proof}

We conclude with a modest application to algebras over uncountable fields.

\begin{cor} \label{uncount} Assume that $k$ is uncountable and that
  $R$ is noncommutative modulo its Jacobson radical. Then $R$ contains
  a free subsemigroup.
\end{cor}

\begin{proof} As noted above, it was proved in \cite{Mak} that a
  division algebra over an uncountable field contains a free
  subsemigroup. The corollary now follows from (\ref{equiv}).
\end{proof}

\begin{Ack} I am grateful to Ken Brown for his very useful comments 
 on an earlier draft of this note, and in particular for his
 suggestion to (and how to) reduce to the affine case in (\ref{PI}).
\end{Ack}

%%%%%%%%%%%%%%%%%%%%%%%%%%%%%%%%%%%-End-Body-%%%%%%%%%%%%%%%%%%%%%%%%%%%%%%%%%%

\end{document}